\documentclass[10pt,fleqn]{article}

\usepackage{extarrows}

\usepackage{enumerate}
\usepackage{amsmath,amsthm,amssymb,amscd}

\usepackage{graphicx}
\usepackage{eepic,color,colordvi,amscd}
\usepackage{subcaption}

\newtheorem{theorem}{Theorem}
\newtheorem{lemma}{Lemma}

\newtheorem{corollary}{Corollary}

\newtheorem{proposition}{Proposition}

\usepackage{graphics}
\usepackage{verbatim}
\makeatletter

\newcommand{\Rmnum}[1]{\expandafter\@slowromancap\romannumeral #1@}
\makeatother

\begin{document}
\title{Every 2-connected, cubic, planar graph with faces of size at most 6 is Hamiltonian}
\author{Sihong Shao\footnotemark[1]
\and Yuxuan Wu\footnotemark[2]}

\renewcommand{\thefootnote}{\fnsymbol{footnote}}
\footnotetext[1]{CAPT, LMAM and School of Mathematical Sciences, Peking University, Beijing 100871, China. Email: \texttt{sihong@math.pku.edu.cn}}
\footnotetext[2]{School of Mathematical Sciences, Peking University, Beijing 100871, China. Email: \texttt{snrhzn@stu.pku.edu.cn}}

%
%
%
%
\maketitle

{\small {\textbf{Abstract}:}  We prove that every 2-connected, cubic, planar graph with faces of size at most 6 is Hamiltonian, and show that the 6-face condition is tight. Our results push the connectivity condition of the Barnette-Goodey conjecture to the weakest possible.

		 }
		
	\section{Introduction}
	
	Barnette's conjecture states that every 3-connected, planar, cubic, bipartite graph is Hamiltonian and has stood for decades without being proved or disproved. The Barnette-Goodey conjecture, a different and possibly weaker version of Barnette's conjecture, has been proved with computer assistance recently \cite{kardos} and replaces the bipartite condition with a face condition, namely every face should be of size at most 6. The main result of this work pushes the connectivity condition of the Barnette-Goodey conjecture to the weakest possible in view of the fact that every Hamiltonian cycle is 2-connected.

	\begin{theorem} \label{t1}
		Every 2-connected, planar, cubic graph with faces of size at most 6 is Hamiltonian.
	\end{theorem}
	
	Meanwhile, Figure~\ref{fcounter} presents a 2-connected, planar, non-Hamiltonian, cubic graph with faces of size at most 7, thereby implying that the 6-face condition of Theorem \ref{t1} is tight.
	
	
		\begin{figure}[h]
		\centering 
		\includegraphics[scale=0.55]{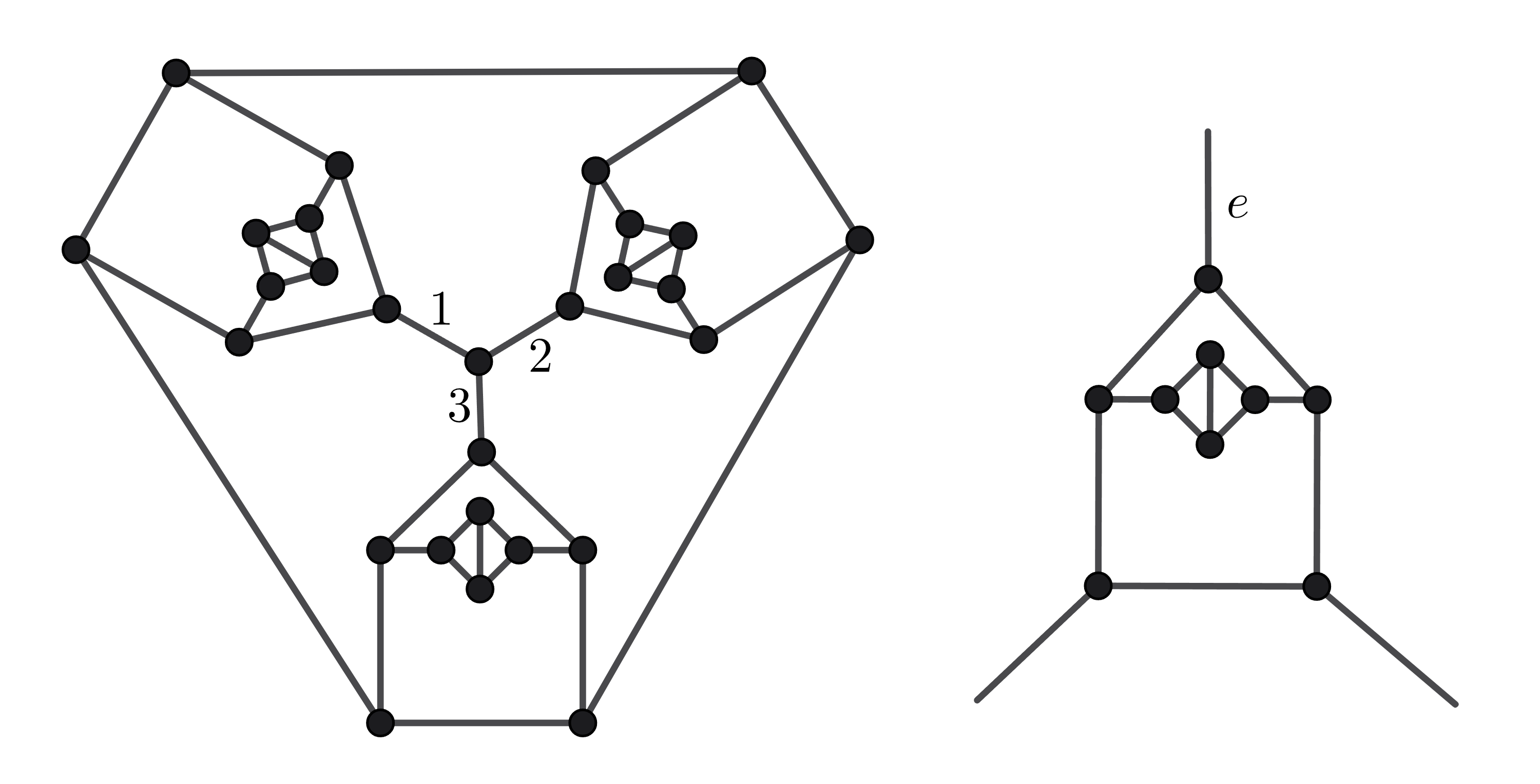}
		\caption{The 6-face condition of Theorem~\ref{t1} is tight: An example. The larger Graph $G$ in the left is a 2-connected, planar, non-Hamiltonian, cubic graph with faces of size at most 7. To see $G$ is indeed non-Hamiltonian, it suffices to use the fact that every Hamiltonian cycle passing through the structure in the right must use the edge $e$. If $G$ is Hamiltonian, then its Hamiltonian cycle must use the edges $1$, $2$, and $3$ in the left plot, which results in a contradiction.}
		\label{fcounter}
	\end{figure}


	The significance and motivation of this work can be explained as follows. For Barnette’s conjecture, every possible way of weakening the condition imposed on the graph, e.g., replacing 3-connectivity with 2-connectivity, or removing the planarity or bipartite condition, renders the conjecture false, and all the decision problems of Hamiltonicity of graphs with the weaker conditions turn out to be NP-complete  \cite{GJT,ANS}. That is, if Barnette's conjecture had been true, 
	then each condition imposed on the graph would be extremely tight. Along a similar research line for the Barnette-Goodey conjecture, although it has already been proved true, it is natural and still meaningful to explore to which extent the imposed condition can be weakened. This work can be regarded as a first attempt in this line. Obviously, there are more directions worth exploring remaining open. For example, whether the Barnette-Goodey conjecture can be strengthened in the face condition,
	or whether Theorem~\ref{t1} can possibly be further strengthened by replacing the planarity condition with something weaker.

%

We use standard notations on graphs throughout the paper. All graphs considered in this paper are simple, undirected and finite.
The rest of the paper is organized as follows. Section~\ref{sec:lemma} provides several lemmas for proving Theorem~\ref{t1}
and Section~\ref{sec:proof} gives the proof.

	\section{Preliminaries and lemmas}
	\label{sec:lemma}
	
	Let $G=(V(G),E(G))$ be a graph. For two non-intersecting subsets of $V(G)$, $V_1,V_2\subseteq V(G), V_1\bigcap V_2=\emptyset$, and an edge subset $S\subseteq E(G)$, we let $G\backslash S=(V(G),E(G)\backslash S)$ and $E(V_1,V_2)$ 
	collect all edges having one endpoint in $V_1$, and the other in $V_2$. For convenience, we use $G-V_1$ to denote the induced subgraph $G[V\backslash V_1]$.  Below we list some facts and lemmas used in proving Theorem~\ref{t1}. 
	

	\setcounter{proposition}{1}
	\begin{proposition} \label{p1}
		(Proposition 2.31 of \cite{OPT}) Let $G$ be a 2-connected graph with a planar embedding $\Phi$.
		Then every face is bounded by a circuit, and every edge is on the boundary of
		exactly two faces. Moreover, the number of faces is $|E(G)| - |V(G)| + 2$. 
	\end{proposition}
	
	\setcounter{corollary}{2}
	\begin{corollary} \label{c1}
		For a 2-connected graph $G$ with a planar embedding $\Phi$, if a vertex $x$ of $G$ has degree 3, and $e_1,e_2,e_3$ are the three edges $x$ is incident with, then the three pairs of edges $(e_1,e_2),(e_2,e_3),(e_3,e_1)$ determine three different faces.
	\end{corollary}
	
	\begin{proof}
			\begin{figure}[h]
			\centering
			\includegraphics[scale=0.65]{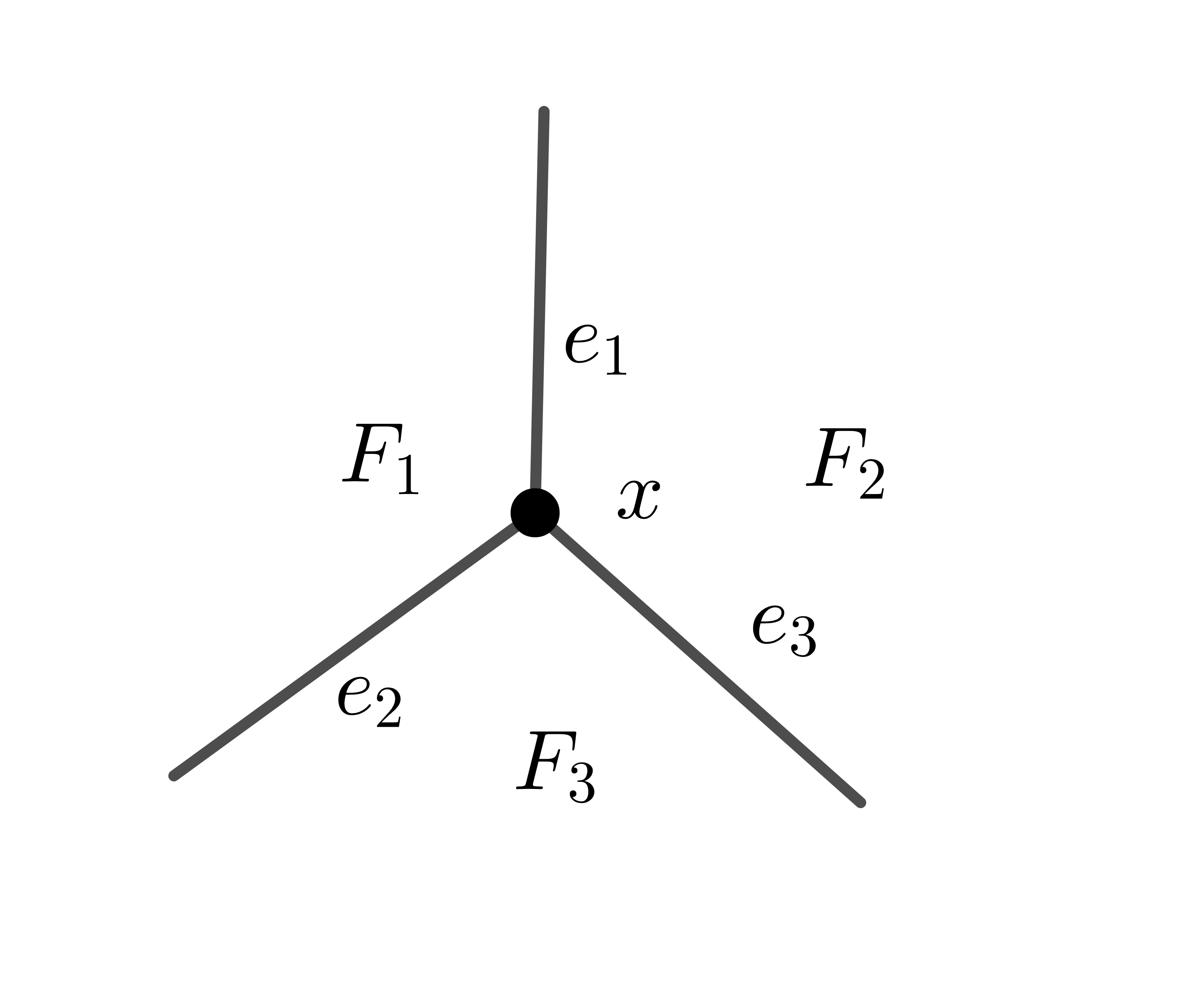}
			\caption{The proof of Corollary~\ref{c1}.}
			\label{f1}
		\end{figure}
	 The edge $e_1$ is on the boundary of two different faces $F_1,F_2$, bounded by different circuits that go through $x$ (see Figure~\ref{f1}). Since $x$ has degree 3, we may assume that $e_1,e_2$ is on the boundary of $F_1$, and $e_1,e_3$ is on the boundary of $F_2$. With the same argument, $e_2$ is on the boundary of two different faces $F_3,F_4$, with $F_3$ having $e_3,e_2$ on its boundary, and $F_4$ having $e_2,e_1$ on its boundary. And $x$ can be incident with at most 3 different faces since its degree is $3$. Thus  $F_1=F_4$ and the result follows.
	\end{proof}

	Given a 2-connected planar graph $G$ and a planar embedding $\Phi$, we know that there is only one face of $\Phi$ that extends to infinity (or equivalently, is unbounded in $\mathbb{R}^2$). We call the unbounded face the \textit{outer face} of $\Phi$.
	For the 3-dimensional sphere $\mathbb{S}_3$ and a point $p$ on it, we know there is a standard continuous bijective projection $f_p: \mathbb{S}_3\backslash \{p\} \mapsto \mathbb{R}^2$. Therefore $f_p^{-1} \circ \Phi$ is an embedding onto $\mathbb{S}_3$, faces of which are in one-to-one correspondence with faces of $\Phi$ (in the sense that the boundary cycles of the corresponding faces are the same), and the pole $p$ in the outer face of $\Phi$ (its counterpart in $\mathbb{S}_3$, to be precise). For any face $F$ of $\Phi$, and its counterpart $f_p^{-1}(F)$ of $f_p^{-1} \circ \Phi$, we can pick a point $q$ in $f_p^{-1}(F)$, and use $f_q$ to project the sphere back to the plane to get another planar embedding $f_q \circ f_p^{-1} \circ \Phi$, the outer face of which is $F$. Therefore we have the following:
	
	\setcounter{proposition}{3}
	\begin{proposition} \label{p2}
		Given a 2-connected planar graph $G$ and a planar embedding $\Phi$,
		for any face $F$ of $\Phi$, there exists a planar embedding $\Phi'$ of $G$ such that the faces of $\Phi$ are in one-to-one correspondence with faces of $\Phi'$ (in the sense that the boundary cycles of the corresponding faces are the same), and $F$ is the outer face of $\Phi'$.
	\end{proposition} 
	
	By Proposition \ref{p2}, when considering a 2-connected planar graph $G$ and an embedding $\Phi$, we can designate a face $F$ of $\Phi$ to be the outer face for one time. This is actually a slight abuse of notation. 
	
	Hereafter when we consider a planar graph $G$, it is always assumed that $G$ is accompanied with a planar embedding $\Phi$. 
	
	\begin{proposition} \label{p3}
		Let $G$ be a 2-connected planar graph. Suppose $C$ is a cycle in $G$ and $V_1$ collects the vertices in the exterior of $C$ ($C$ excluded). Then the induced graph $H_1=G- V_1$ is a 2-connected planar graph, too.
	\end{proposition} 
	\begin{proof}
		Obviously $H_1$ is planar and connected. $\forall x\in V(H_1)$, we are going to prove $H_1-\{x\}$ is connected. That is, we need to find a walk connecting $y$ and $z$ for all vertex pairs of $y, z\in H_1-\{x\}, y\neq z$. Since $G$ is 2-connected, there is a walk $P$ connecting $y,z$ in $G-\{x\}$. If $P$ is contained in the interior of $C$ ($C$ included), we are done. Otherwise suppose $P$ starts from $y$ and ends at $z$. Assume that the first time $P$ goes out from the interior of $C$ is through vertex $w_1\in V(C)$, and the last time $P$ goes into the interior of $C$ is through vertex $w_2\in V(C)$. We modify the part between $w_1,w_2$ to get a walk $P'$ contained in the interior of $C$. If $w_1=w_2$, simply remove the part between $w_1$ and $w_2$; otherwise there must be a walk $P_1$ connecting $w_1,w_2$ in the graph $C- \{x\}$ for $C$ is 2-connected, and thus replacing the part in $P$ between $w_1$ and $w_2$ with $P_1$ gives $P'$.
	\end{proof}
	\setcounter{lemma}{5}
	\begin{lemma} \label{l51}
		Let $G$ be a 2-connected cubic graph. If $G$ is not 3-connected, then $G$ must have a 2-edge-cut $\{e, f\}$. Moreover, $\{e,f\}$ is a matching and $G\backslash \{e, f\}$ has exactly two connected components, each of which shares exactly two vertices with the four endpoints of $e, f$. 
	\end{lemma}
	\begin{proof}
	Since $G$ has at least four different vertices and is not 3-connected, there are two vertices $x, y$ whose removal leaves the remaining graph disconnected. Pick two different connected components $A,B$ of $G-\{x,y\}$ (we have not yet excluded the possibility that there are more than two connected components in $G-\{x,y\}$). Since $G$ is 2-connected, there must be at least one edge between $x$ and $A$, otherwise the removal of $y$ leaves the graph disconnected. The same argument goes for $y$ and $A$ as well as $x, y$ and $B$. Since the degree of $x$ is $3$, we have either $|E(\{x\},A)|=1$ or $|E(\{x\},B)|=1$. Without loss of generality, we assume $|E(\{x\},A)|=1$, and denote $E(\{x\},A)=\{e\}$. If $|E(\{y\},A)|=1$, assuming $E(\{y\},A)=\{f\}$, then $\{e,f\}$ is a 2-edge-cut (see the left plot of Figure~\ref{f51}). Otherwise, there are two edges between $y$ and $A$ and one edge between $y$ and $B$, denoted by $f$, for $y$ has degree $3$,
	and thus $x$ and $y$ must be not adjacent, thereby implying $\{e,f\}$ is a 2-edge-cut (see the middle plot of Figure~\ref{f51}).  
	
	Now suppose a 2-edge-cut $\{e,f\}$ is not a matching and shares a vertex $z$. Since $z$ has degree $3$, it is incident with an edge $(z, r)$ other than $e, f$, and thus $G-\{z\}$ is disconnected, yielding a contradiction to the 2-connectivity of $G$ (see the right plot of Figure~\ref{f51}).  Let $e=xu, f=yv$. 
The connectivity of $G$ implies that any connected component in $G\backslash \{e,f\}$ must share at least one vertex with $\{x, y, u, v\}$. Suppose one of such components, say $A$, share only one vertex with $\{x, y, u, v\}$, say $x$. Then $G-\{u\}$ is disconnected, a contradiction to 2-connectivity of $G$. Therefore every connected component in $G\backslash \{e,f\}$ must share at least two vertices with $\{x,y,u,v\}$. As a result, there must be exactly two components, each of which shares exactly two vertices with $\{x,y,u,v\}$.
		\begin{figure}[h] 
			\centering
			\includegraphics[scale=0.6]{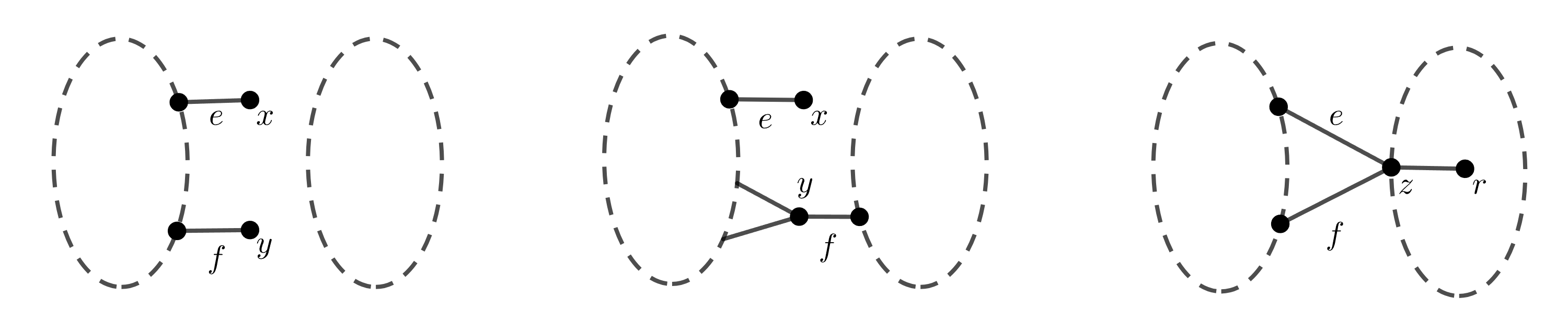} 
			\caption{The proof of Lemma~\ref{l51}.}
			\label{f51}
		\end{figure}
	\end{proof}
	
	\begin{lemma} \label{l52}
		There is no graph $G$ satisfying the following properties: 
		(1) $G$ is 2-connected planar with faces of size no more than 6;
		(2) The outer face of $G$ is a quadrilateral consisting of four vertices, say $xyzw$, in cyclic order, and $x, y$ have degree 2;
		(3) Each vertex in $V(G)\backslash \{x,y\}$ has degree 3.
	\end{lemma}
			\begin{figure}[h] 
			\centering
			\includegraphics[scale=0.8]{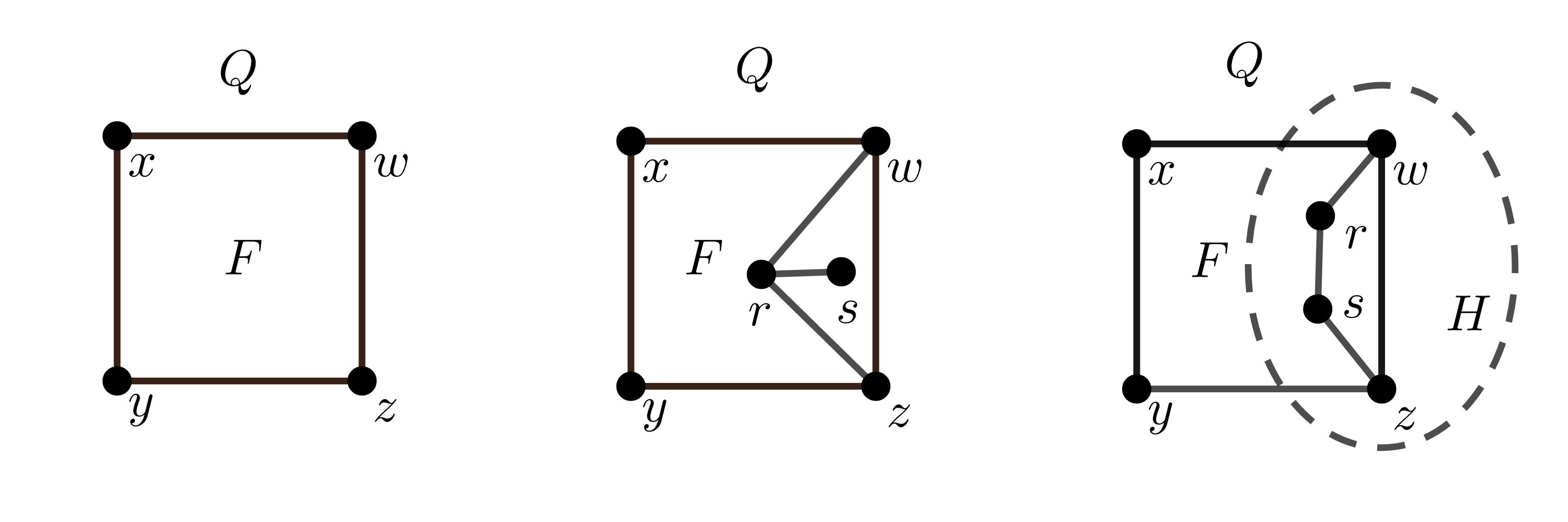} 
			\caption{The proof of Lemma~\ref{l52}.}
			\label{f52}
		\end{figure}
	\begin{proof}
	Suppose for contradiction that there exists at least one graph satisfying all above-mentioned properties, and let $G$ being such a graph with the smallest number of vertices. By Proposition \ref{p1}, the edge $xy$ is incident to two faces: one is the outer quadrilateral $Q$, and the other we denote by $F$ (see the left plot of Figure~\ref{f52}). Since both $x$ and $y$ have degree 2, the boundary of $F$ must contain edges $xy, yz, xw$. After removing $xy, yz, xw$, the boundary of $F$ is then a $zw$-path, denoted by $P$. Since $F$ has size at most 6, $P$ has at most 3 edges. Below we proceed in three cases.

	$-$ $P$ has one edge, i.e., $P=zw$. Then, the edges $zw, zy$ determine two different faces $F, Q$, while Corollary~\ref{c1} requires that $zw, zy$ share only one face, a contradiction; 
	

	
	$-$ $P$ has two edges, i.e., $P=zrw$ (see the middle plot of Figure~\ref{f52}). Since $r$ has degree 3, it is incident with a vertex $s$ other than $z, w$, but then $G-\{r\}$ is disconnected (or otherwise the degree restriction must be violated), a contradiction to the 2-connectivity of $G$; 
	
	$-$ $P$ has three edges, i.e., $P=zsrw$ (see the right plot of Figure~\ref{f52}). Consider the graph $H=G-\{x,y\}$, which can also be obtained by removing the vertices in the exterior of the cycle $zsrwz$. Accordingly, $H$ is planar with faces of size at most 6 and 2-connected (by Proposition~\ref{p3}). Moreover the outer face of $H$ is the quadrilateral $zsrw$, with $z,w$ having degree 2, and all vertices in $V(H)\backslash\{z,w\}$ have degree 3.  That is, $H$ satisfies the properties (1)-(3), but has fewer vertices than $G$,  a contradiction to the minimality of $G$.
	\end{proof}
	
	\begin{lemma} \label{l53}
		Suppose $G$ is a graph satisfying (1) $G$ is 2-connected planar with faces of size at most 6; (2) The outer face of $G$ is a quadrilateral consisting of four vertices, say $xyzw$, in cyclic order, and $x, z$ have degree 2; and (3) Each vertex in $V(G)\backslash \{x, z\}$ has degree 3. Then $G$ has a Hamiltonian path with $x$ and $z$ being its endpoints.
	\end{lemma}
			\begin{figure}[h] 
			\centering
			\includegraphics[scale=0.8]{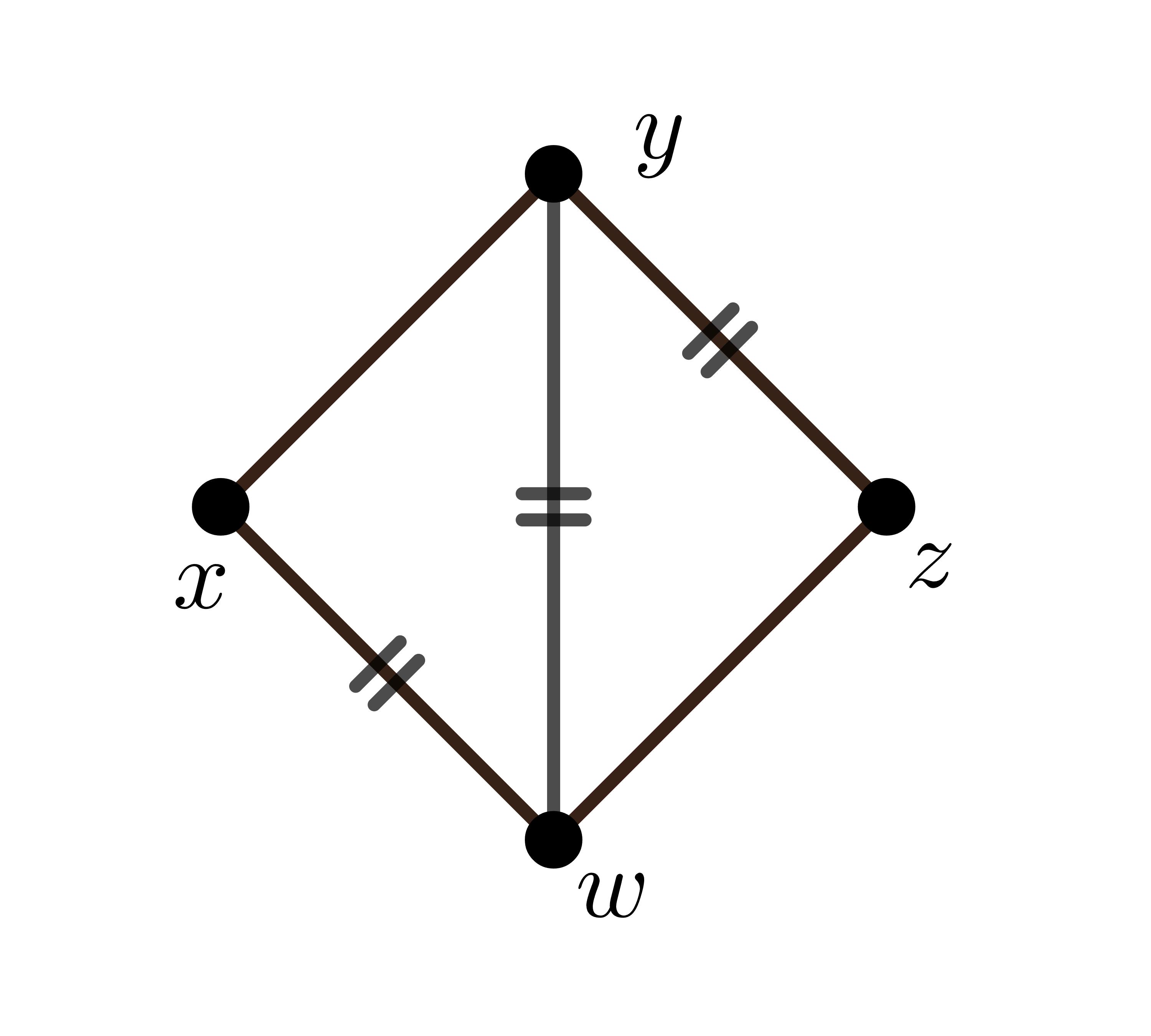}
			\caption{The proof of Lemma~\ref{l53}: The smallest graph $G$ and its Hamiltonian path marked with double dash.}
			\label{f531}
		\end{figure} 
	\begin{proof}
		We prove by induction on $|V(G)|$, the number of vertices of the underlying graph $G$. It can be readily verified that Figure~\ref{f531} shows the smallest graph satisfying the properties (1)-(3) with $|V(G)|=4$, which is unique under graph isomorphism, and the required Hamiltonian path is marked with double dash. Now suppose any graph $G$ satisfying the properties (1)-(3) and $|V(G)|\leq k (k\geq 4)$ has a Hamiltonian path satisfying the requirement, and consider a graph $G_1$ satisfying (1)-(3) and $|V(G_1)|=k+1$ (if there is no such graph $G_1$, then it is trivial to extend the result to $k+1$, so we may always assume there exists at least one such $G_1$). We are going to prove $G_1$ has a Hamiltonian path meeting the requirement. 

		\begin{figure}[h] 
			\centering
			\includegraphics[scale=0.8]{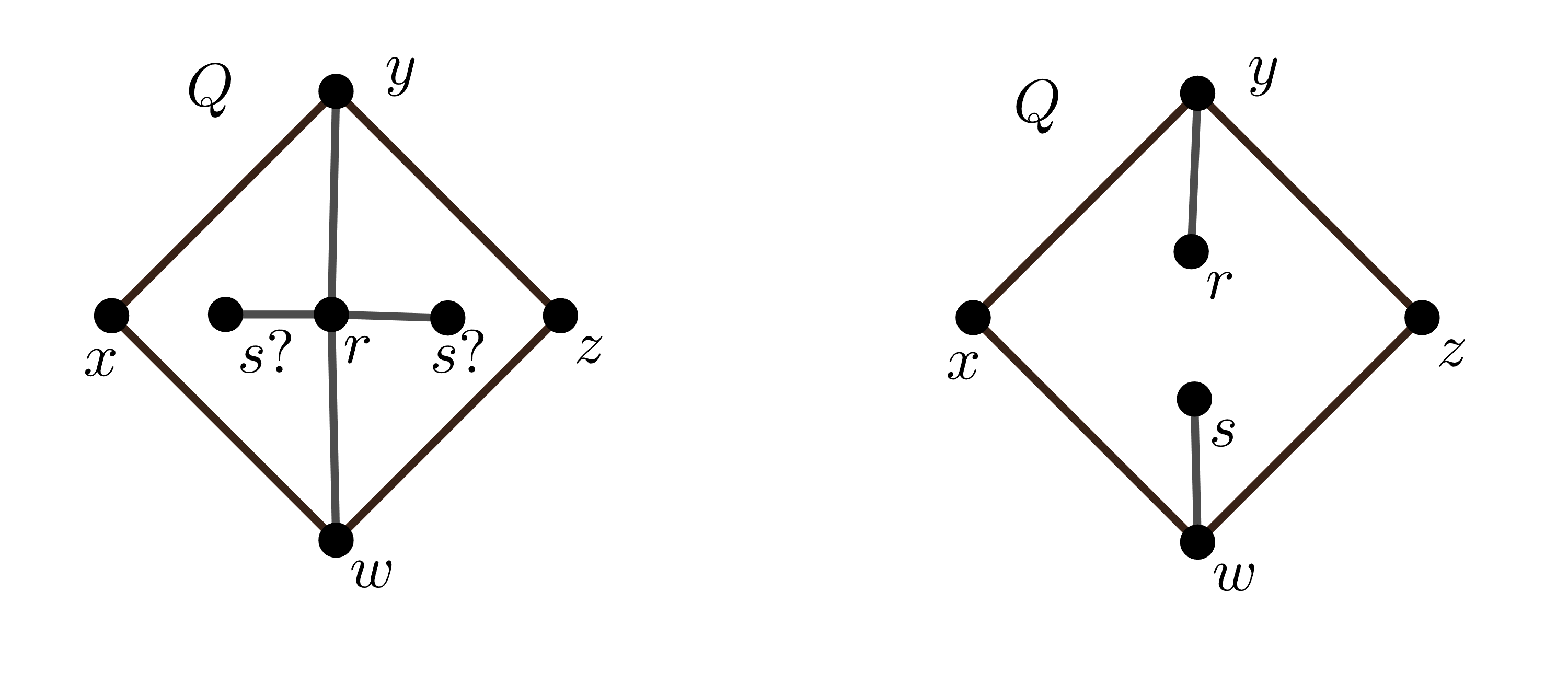}
			\caption{The proof of Lemma~\ref{l53}: Proving that $y,w$ are not adjacent, and the neighbors of $y$ and $w$ other than $x,z$ are different.}
			\label{f532}
		\end{figure}
		
		Assume that the outer face of $G_1$ is $Q=xyzw$ (vertices in cyclic order), and the vertices $x, z$ are of degree 2. If $y$ is adjacent to $w$, then there can be no more vertices other than $x,y,z,w$ due to the restriction on degree. So $|V(G_1)|=4<k+1$, a contradiction. Accordingly, suppose $y$ is adjacent to vertex $r$ other than $x,z$. If $w$ is adjacent to $r$, then $r$ has a neighbor $s$ other than $y,w$ and  no path connects $s$ and Q. But then $G_1-\{r\}$ is disconnected, or otherwise the degree restriction must be violated (see the left plot of Figure~\ref{f532}). This is a contradiction to the 2-connectivity of $G$. Hence $w$ has a neighbor $s$ other than $x,z$ and $s\neq r$ (see the right plot of Figure~\ref{f532}).
		
		\begin{figure}[h] 
			\centering
			\includegraphics[scale=0.67]{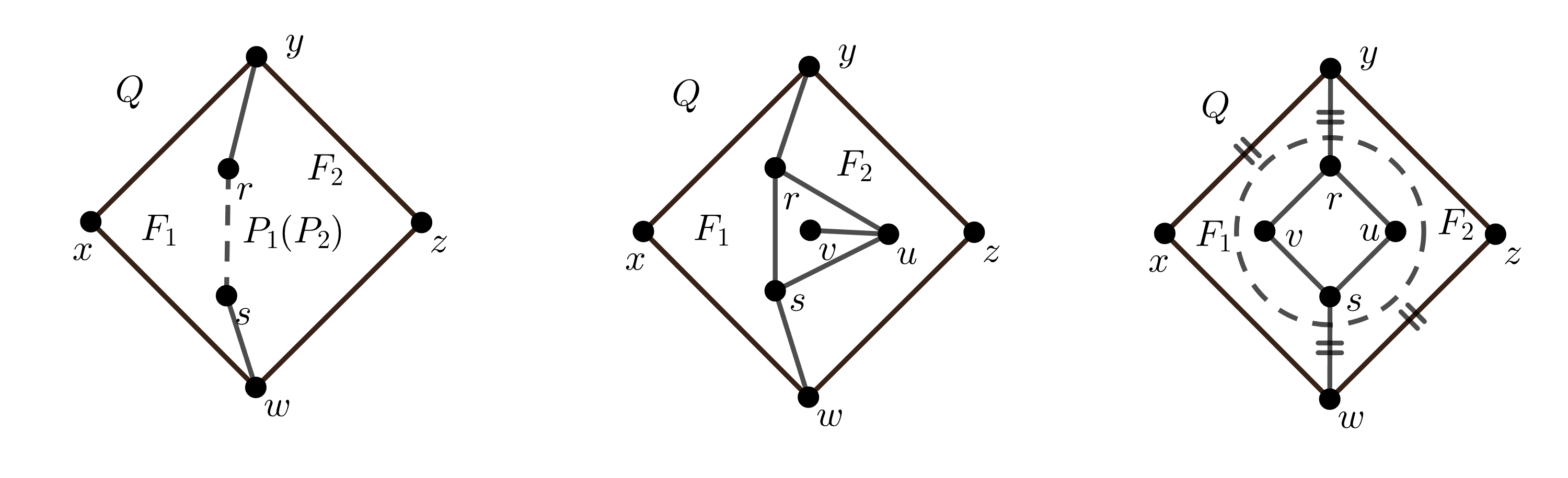}
			\caption{The proof of Lemma~\ref{l53}: Constructing a Hamiltonian path with $x, z$ being its endpoints.}
			\label{f534}
		\end{figure}
		
		 By Proposition \ref{p1}, the edge $xy$ is incident to two faces: one is the outer quadrilateral $Q$ and the other we denote by $F_1$, while the edge $yz$ is also incident to two faces: one is also $Q$ and the other we denote by $F_2$. By Corollary \ref{c1} the edges $yr,ws$ are on the boundary of both $F_1$ and $F_2$. The boundary of $F_1$, with edges $wx,xy,yr,ws$, removed, is a $rs$-path $P_1$. Similarly, the boundary of $F_2$, with edges $wz,zy,yr,ws$ removed, is a $rs$-path $P_2$ (see the left plot of Figure~\ref{f534}). If $P_1=P_2$, then the edge $yr$, and the edge in $P_1$ having $r$ as an endpoint determine two faces $F_1,F_2$, a contradiction to Corollary~\ref{c1}. Therefore $P_1\neq P_2$. Since $F_1,F_2$ are of size at most 6, $P_1$ and $P_2$ are of length at most 2. Suppose one of them, say $P_1$, is of length 1, i.e., $P_1=rs$, and then $P_2$ is of length 2. Let $P_2=rus$. Then $u$ has a neighbor $v$ other than $r,s$. Therefore $G_1-\{u\}$ is disconnected, a contradiction to the 2-connectivity of $G$, or otherwise the degree restriction must be violated (see the middle plot of Figure~\ref{f534}). So $P_1,P_2$ are both of length 2. Assume $P_1=rvs$, $P_2=rus$. Consider the graph $H = G_1-\{x,y,z,w\}$. Obviously, $H$ can be viewed as removing all vertices in the exterior of the cycle $rusvr$ (see the right plot of Figure~\ref{f534}). By Proposition~\ref{p3}, $H$ is a 2-connected planar graph. Also, each face of $H$ is of size at most 6. The outer face of $H$ is the quadrilateral $rusv$, with $r,s$ of degree 2. And all vertices in $V(H)\backslash\{r,s\}$ are of degree 3. That is, $H$ satisfies the properties (1)-(3) and $|V(H)|= k+1-4\leq k$. Applying the induction hypothesis to $H$ gives a Hamiltonian path $P_H$ in $H$ with $r,s$ being its endpoints, and thus $xyrP_Hswz$ is the required Hamiltonian path in $G_1$. The induction is complete.
	\end{proof}
	
	\section{Proof of the main result}
	\label{sec:proof}
	
	\begin{figure}[h] 
		\centering
		\includegraphics[scale=0.7]{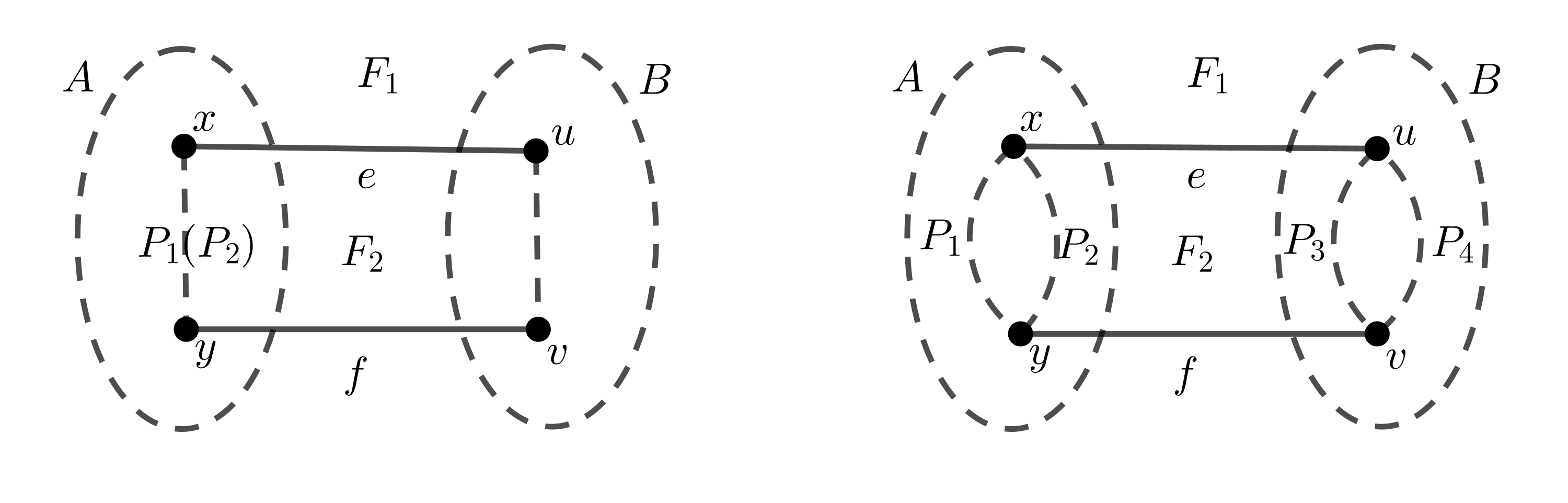}
		\caption{The proof of Theorem~\ref{t1}: Proving by contradiction that $P_1\neq P_2, P_3\neq P_4$.}
		\label{f11}
	\end{figure}

	Suppose there is a 2-connected, cubic, planar graph $G$ with faces of size no more than 6. Since $G$ is simple, we have $|V(G)|\geq 4$. If $G$ is 3-connected, then it directly follows \cite{kardos} that $G$ is Hamiltonian. Below it suffices to assume that $G$ is not 3-connected.
	By Lemma~\ref{l51}, $G$ has a 2-edge-cut $\{e, f\}$ with $e=xu, f=yv$, $\{e, f\}$ is a matching,
	and $G\backslash \{e,f\} $ has exactly two connected components $A, B$. 
	Without loss of generality,  we may assume $x,y\in V(A)$ and $u,v\in V(B)$. By Proposition \ref{p1}, the edge $e$ is incident to two different faces, say $F_1,F_2$. For all $i\in \{1,2\}$, the boundary of $F_i$, after removing $e$, is a path connecting $A$ and $B$. Since $\{e,f\}$ is a 2-edge-cut, such a path must contain the edge $f$. Therefore $f$ is incident to both $F_1,F_2$, and by Proposition \ref{p1} $f$ is incident to no other faces. We designate $F_1$ to be the outer face. Now, the boundary of $F_1$, while removed of the two edges $e, f$, consists of a $xy$-path, denoted by $P_1$, and a $uv$-path, denoted by $P_4$. Similarly, the boundary of $F_2$, while removed of the two edges $e,f$, consists of a $xy$-path, denoted by $P_2$, and a $uv$-path, denoted by $P_3$. Suppose $P_1=P_2$, then the edge $e$ and the edge in $P_1 (P_2)$ with $x$ being one of its endpoints determine two faces $F_1,F_2$ (see the left plot of Figure~\ref{f11}), a contradiction to Corollary \ref{c1} since $x$ has degree 3. Therefore $P_1\neq P_2$. By a similar argument, $P_3\neq P_4$ (see the right plot of Figure~\ref{f11}). We claim that, for all $i\in\{1, 2, 3, 4\}$, $P_i$ has exactly two edges (see the right plot of Figure~\ref{f12}). 
	
		\begin{figure}[h] 
		\centering
		\includegraphics[scale=0.7]{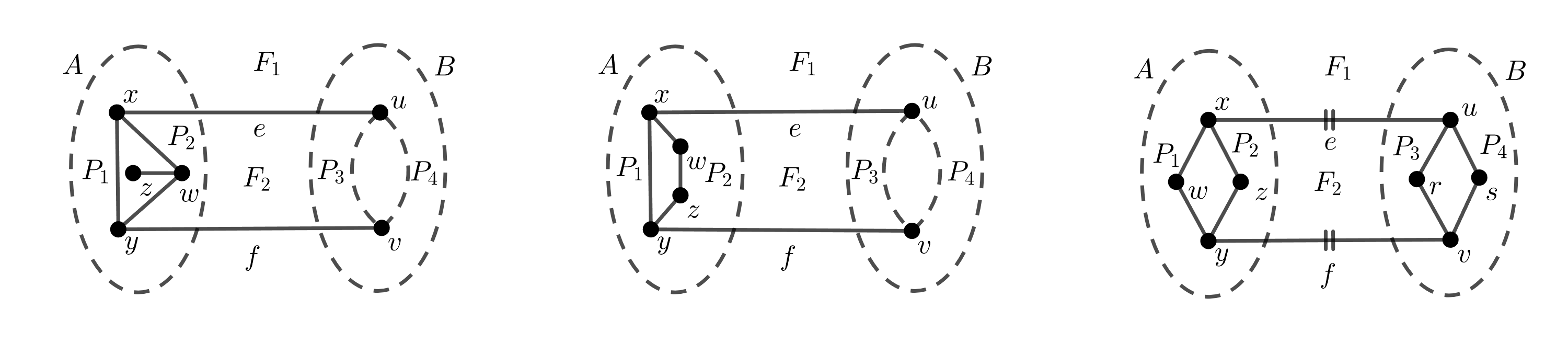}
		\caption{The proof of Theorem~\ref{t1}: Constructing a Hamiltonian cycle via the 2-edge-cut $\{e, f\}$.}
		\label{f12}
	\end{figure} 
	
Suppose $P_1$ is of length one, i.e., $P_1=xy$. Since $P_1\neq P_2$, $P_2$ has length at least 2. On the other hand, since $F_2$ is of size at most 6, $P_2$ has length at most 3. If $P_2$ has length 2, i.e., $P_2=xwy$, then $w$ has a neighbor $z$ other than $x,y$ (see the left plot of Figure~\ref{f12}). But then $G- \{w\}$ is disconnected (otherwise the degree restriction must be violated), a contradiction to the 2-connectivity of $G$. Suppose $P_2$ has length 3, i.e., $P_2=xwzy$. Consider the subgraph $A$, which can also be viewed to be obtained by removing all the vertices in the exterior of the cycle $xyzwx$ (see the middle plot of Figure~\ref{f12}). By Proposition \ref{p3}, $A$ is 2-connected and planar. It can also be seen that every face of $A$ has size at most 6. Additionally, the outer face of $A$ is the quadrilateral $xyzw$, with $x,y$ of degree 2. Also, all vertices in $V(A)\backslash\{x,y\}$ are of degree 3 in $A$. By Lemma \ref{l52}, this is a contradiction. Thus $P_1$ has length at least 2. If $P_2$ has length one, we re-designate $F_2$ to be the outer face, apply a similar argument and arrive at the same contradiction. By a similar argument, the lengths of $P_3,P_4$ are at least 2. Therefore the claim holds because $F_1,F_2$ are of size at most 6.

	Let $P_1=xwy, P_2=xzy, P_3=urv, P_4=usv$ (see the right plot of Figure~\ref{f12}). Consider the subgraph $A$, which can also be viewed to be obtained by removing all vertices in the exterior of the cycle $xzywx$. By Proposition~\ref{p3}, $A$ is 2-connected and planar. Moreover, each face of $A$ has size at most 6. The outer face of $A$ is the quadrilateral $xzyw$. The vertices $x,y$ are of degree 2 in $A$, and all other vertices in $A$ are of degree 3. By Lemma~\ref{l53}, $A$ has a Hamiltonian path $H_A$ with $x,y$ being its endpoints. By a similar argument, $B$ has a Hamiltonian path $H_B$ with $u,v$ being its endpoints, too. Therefore, $G$ has a Hamiltonian cycle $xH_AyvH_Bux$.

	\section*{Acknowledgements}
	This work was funded by the National Key R \& D Program of China (No.~2022YFA1005102) and
the National Natural Science Foundation of China (Nos.~12325112, 12288101). 
The authors are grateful to the useful discussions with Qiming Fang.

%


\end{document}